\documentclass{amsart}
\usepackage[english]{babel}
\title{Observation of vibrating systems at different time instants}
\author{Ambroise Vest}
\address{Institut de Recherche Math\'ematique Avanc\'ee,  Universit\'e de Strasbourg\\7 rue Ren\'e Descartes, 67084 Strasbourg C\'edex, France}
\email{ambroise.vest@math.unistra.fr}
\date{\today}
\usepackage{amsmath, amssymb, amsthm}
\usepackage{graphicx}
\usepackage{empheq}
\newcommand{\C}{\mathbb{C}}
\newcommand{\Q}{\mathbb{Q}}
\newcommand{\R}{\mathbb{R}}
\newcommand{\N}{\mathbb{N}}
\newcommand{\Z}{\mathbb{Z}}

\newcommand{\om}{\omega}

\newcommand{\ud}{\,\mathrm{d}}

\newtheorem{thm}{Theorem}[section]
\newtheorem{lem}[thm]{Lemma}
\newtheorem{prop}[thm]{Proposition}

\theoremstyle{definition}
\newtheorem{deff}[thm]{Definition}

\theoremstyle{remark}
\newtheorem*{rk}{Remark}

\theoremstyle{remark}
\newtheorem*{rks}{Remarks}

\theoremstyle{remark}

\theoremstyle{remark}

\theoremstyle{remark}

\theoremstyle{remark}

\begin{document}
\maketitle
\begin{abstract}
In this paper, we obtain new observability inequalities for the vibrating string. 
This work was motivated by a recent paper by A. Szij\'art\'o and J. Heged\H us  in which the authors ask 
the question of determining the initial data  by only knowing the position of the string at two distinct time instants.
The choice of the observation instants is crucial and the estimations rely on the Fourier series expansion of the solutions and results of Diophantine approximation.
\end{abstract}
\section{Introduction}
Let $q$ be a nonnegative number. The small transversal vibrations of a string of length $\pi$ fixed at its two ends satisfy
\footnote{
The quantity $y=y(t,x)$ is the height of the string at time $t$ and abscissa $x$ while $y(t)$ stands for the map $y(t,\cdot)$. The choice of $\pi$ for the length of the string is made in order to simplify the writing in the expansion of the solutions in Fourier series.
}
\begin{equation} \label{string}
\begin{cases}
y''-y_{xx}+qy=0 & \text{in } \R\times(0,\pi),\\
y=0 & \text{in }\R\times\{0,\pi\},\\
y(0)=y_0, \quad y'(0)=y_1 & \text{in }(0,\pi).
 \end{cases}
\end{equation}

The obtaining of observability inequalities for the vibrating string and for oscillating systems in general has been the object of many works. Indeed, observability being dual to controllability (cf. D. L. Russell \cite{Russell1978}), it is often a starting point to obtain controllability results (see e.g., J.-L. Lions \cite{Lions1988}, A. Haraux \cite{Haraux1990}). 
A useful tool to obtain such inequalities is the Fourier series expansion of the solutions (cf. V. Komornik and P. Loreti \cite{KL05}).

Among all the different ways to observe the system \eqref{string}, \emph{pointwise observation} has been widely studied (see e.g., J.-L. Lions \cite{Lions1992}, A. Haraux \cite{HarauxPonct}). It consists in getting estimations of the form 
\begin{equation*}
\|(y_0,y_1)\|_{\mathcal{I}}\le c\|y(\cdot,\xi)\|_{\mathcal{O}}.
\end{equation*}
The main difficulties are the choice of the norms for the initial data $\|\cdot\|_{\mathcal{I}}$ as for the observation $\|\cdot\|_{\mathcal{O}}$ and the choice of a \emph{strategic point} $\xi$ in the domain. These particular points can be characterized by some of their arithmetical properties (see e.g., A. G. Butkovskiy \cite{But1979}, V. Komornik and P. Loreti \cite{KL11}).

Following a recent paper of A. Szij\'art\'o and J. Heged\H us \cite{SH12}, we focus  on a \emph{pointwise-in-time observation}. Such type of observation seems to have been studied at first by A. I. Egorov \cite{Egorov2008} and  L. N. Znamenskaya \cite{Znamen2010}. Given two norms, one for the initial data $\|\cdot\|_{\mathcal{I}}$ and one for the observation $\|\cdot\|_{\mathcal{O}}$, the objective is to find two times $t_0$ and $t_1$ such that 
\begin{equation}\label{generalobservation}
\|(y_0,y_1)\|_{\mathcal{I}}\le c(\|y(t_0)\|_{\mathcal{O}}+\|y(t_1)\|_{\mathcal{O}})
\end{equation}
From a practical point of view, such an inequality means that only knowing the position of the whole system at two different instants, we are able to recover the initial data $y_0$ and $y_1$. 
\begin{deff}
A pair $(t_0,t_1)$ of real numbers such that the  observability inequality \eqref{generalobservation} holds is called a \emph{strategic pair (for \eqref{generalobservation})}.
\footnote{In particular, this notion depends on the the norms in the left and right members.
}
\end{deff}

The main idea of this paper is the following : depending on how the quantity 
\begin{equation*}
\frac{t_0-t_1}{\pi}
\end{equation*}
is approximable by rational numbers,  such pointwise-in-time observability inequalities hold. The main tools are the explicit expansion of the solutions in Fourier series and classical results of Diophantine approximation. 

\bigskip
Let us describe the organization of the paper and state (informally) the main results. 

In section 2, after recalling the definition of adapted functional spaces to study the well-posedness of \eqref{string}, we reformulate the observation problem  in this setting. These spaces, denoted by $D^s$ ($s \in \R)$, correspond essentially to the domain of $-\Delta^{s/2}$. Then, we may chose two real numbers $r$ and $s$ such that $\|\cdot\|_{\mathcal{I}}=\|\cdot\|_{D^s}$ and $\|\cdot\|_{\mathcal{O}}=\|\cdot\|_{D^r}$.

In section \ref{classicalstring}, we investigate the observation of the \emph{classical string} (i.e. $q=0$). We prove (see Theorem \ref{MRclassicalstring}) the following result :
\begin{quote} 
\emph{
Assume that $r-s\ge 1$. Then, there exist strategic pairs. Moreover, if the inequality is strict, then almost all pairs are strategic. This result is optimal in the sense that there cannot be any strategic pair if $r-s<1$.
}
\end{quote}

In section \ref{more}, we prove that the difference $r-s$ (see Theorem \ref{moreobservations}) can be reduced by adding further observations.

In section \ref{loadedstring}, we focus on the \emph{loaded string} (i.e. $q>0$). First we recall the main result of \cite{SH12} in Theorem \ref{thmSH}, which states essentially that if $(t_0-t_1)/\pi$ is a \emph{rational} number along with another hypothesis, then $(t_0,t_1)$ is a strategic pair with $r-s=1$. After analyzing the occurence of such pairs under the above hypotheses in Proposition \ref{density}, we use another method to obtain new observability inequalities. We can state the following result (see Theorem \ref{MRloadedstring}):
\begin{quote}
\emph{
Assume that $r-s=1$. If $(t_0,t_1)$ is a strategic pair for the classical string, then it is also a strategic pair for the loaded string provided that $q$ is sufficiently small.
}
\end{quote}

Finally, in section \ref{other}, we extend our method to the vibrating beam and rectangular plates. 
\section{Problem setting and notations}\label{setting} 
Let us recall the construction of some useful functional spaces related to the above problem (see e.g., \cite[pp. 7--11]{Kom94}, \cite[pp. 335--340]{Brezis2}).
The functions $\sin(kx),\, k=1,2,\ldots$ form an orthogonal and dense system in $L^2(0,\pi)$. 
We denote by $D$ the vector space spanned by these functions and for $s\in \R$, we define an euclidean norm on $D$ by setting
\begin{equation*}
\Big\|\sum_{k=1}^{\infty}c_k\sin(kx)\Big\|_s^2:=\sum_{k=1}^{\infty}k^{2s}|c_k|^2.
\end{equation*}
The space $D^s$ is defined as the completion of $D$ for the norm $\|.\|_s$.
Then, $D^0$ coincide with $L^2(0,\pi)$ with equivalent norms and
more generally, it is possible to prove that for $s>0$, 
\begin{equation*}
D^s=\Big\{f\in H^s(0,\pi) : f^{(2j)}(0)=f^{(2j)}(\pi)=0,\quad \forall \, 0\le j\le\Big[\frac{s-1}{2}\Big]\Big\}.
\end{equation*}
Identifying $D^0$ with its own dual, $D^{-s}$ is the dual of $D^s$. For example,
\begin{equation*}
D^0=L^2(0,\pi), \quad D^1=H^1_0(0,\pi) \quad \text{and} \quad D^{-1}=H^{-1}(0,\pi)
\end{equation*}
with equivalent norms. 

Now, we recall a well-posedness result for the problem \eqref{string} via an expansion of the solutions in Fourier series. We set
\begin{equation*}
\om_k:=\sqrt{k^2+q},\qquad k=1,2,\ldots
\end{equation*}
\begin{prop}
Let $s \in \R$. For all initial data $y_0 \in D^s$ and $y_1 \in D^{s-1}$, the problem \eqref{string} admits a unique solution $y \in C(\R,D^s) \cap C^1(\R,D^{s-1}) \cap C^2(\R,D^{s-2})$ given by
\begin{equation} \label{solution}
 y(t,x)=\sum_{k=1}^{\infty}(a_k e^{i\om_kt}+b_k e^{-i\om_kt})\sin{kx},
\end{equation}
where the complex coefficients $a_k$ and $b_k$ satisfy
\footnote{
$A\asymp B$ means that there are two positive constants $c_1$ and $c_2$ such that $c_1 B\le A\le c_2 B$.
}
\begin{equation} \label{conservation}
 \|y_0\|_s^2+\|y_1\|_{s-1}^2\asymp\sum_{k=1}^{\infty}k^{2s}(|a_k|^2+|b_k|^2).
\end{equation}
\end{prop}

The observability problem that we are going to investigate in this paper is the following. Given two real numbers $r$ and $s$ such that $s\le r$, we ask wether or not there are two times $t_0$ and $t_1$ such that 
\begin{equation} \label{observation2}
 \|y_0\|_s+\|y_1\|_{s-1}\leq c(\|y(t_0)\|_r+\|y(t_1)\|_r)
\end{equation}
for a positive constant $c$, independent of the initial data $(y_0,y_1)\in D^r \times D^{r-1}$.
\section{Observability of the classical string ($q=0$)}\label{classicalstring}
In this paragraph, we assume that $q=0$ in the problem \eqref{string}.
The following statement transforms the observation inequality \eqref{observation2} to a problem of Diophantine approximation.
\begin{prop}\label{thmstring}
The pair $(t_0,t_1)$ is strategic if and only if there is a positive constant $c$ such that
\footnote{
If $x$ is a real number, $\|x\|$ denotes the distance between $x$ and the nearest integer.
}
\begin{equation}\label{estthmstring}
\Big\|\frac{k(t_0-t_1)}{\pi}\Big\| \ge \frac{c}{k^{r-s}}, \qquad k=1,2,\ldots
\end{equation}
\end{prop}
For the proof, we need the following 
\begin{lem} \label{lemmasinus}
Set $x \in \R$. We have
\begin{equation*}
|\sin kx| \asymp \Big\|\frac{kx}{\pi}\Big\|, \qquad k=1,2,\ldots
\end{equation*}
\end{lem}
\begin{proof}[Proof of Lemma \ref{lemmasinus}]
We follow the proof of \cite[Lemma 2.3]{KL11}. Denoting by $m$ the nearest integer from $kx/\pi$,
\begin{equation*}
|\sin kx|=|\sin(kx-m\pi)|=\Big|\sin\Big(\frac{kx}{\pi}-m\Big)\pi\Big|.
\end{equation*}
We notice that $|kx/\pi-m|\pi \le \pi/2$. Hence, using the estimations  $(2/\pi)|t|\le|\sin t|\le |t|$ which hold for 
$|t|\le \pi/2$, we have
\begin{equation*}
\frac{2}{\pi}\Big|\frac{kx}{\pi}-m\Big|\pi\le\Big|\sin\Big(\frac{kx}{\pi}-m\Big)\pi\Big| \le \pi\Big|\frac{kx}{\pi}-m\Big|,
\end{equation*}
i.e.
\begin{equation*}
2\Big\|\frac{kx}{\pi}\Big\| \le |\sin kx| \le \pi \Big\|\frac{kx}{\pi}\Big\|. \qedhere
\end{equation*}
\end{proof}
\begin{proof}[Proof of Proposition \ref{thmstring}]
Using the Fourier series expansion \eqref{solution} of the solutions of \eqref{string} and the estimation \eqref{conservation}, we observe that the left member in \eqref{observation2} is equivalent
\footnote{
in the sense of the symbol $\asymp$
}
to
\begin{equation*}
\sum_{k=1}^{\infty}k^{2s}(|a_k|^2+|b_k|^2)
\end{equation*}
and the right member is equivalent to
\begin{equation*}
\sum_{k=1}^{\infty}k^{2r}(|a_ke^{ikt_0}+b_ke^{-ikt_0}|^2+|a_ke^{ikt_1}+b_ke^{-ikt_1}|^2).
\end{equation*}
Therefore, the observability inequality \eqref{observation2} holds if and only if there exist a positive constant $c'$ such that for all $k=1,2,\ldots$ and all complex numbers $a$ and $b$,
\begin{equation}\label{estmode} 
 k^{2s}(|a|^2+|b|^2)\leq c'k^{2r}(|ae^{ikt_0}+be^{-ikt_0}|^2+|ae^{ikt_1}+be^{-ikt_1}|^2).
\end{equation}

Now, for all $k$, we consider the linear maps $T_k$ in $\C\times\C$ (endowed with its usual euclidean norm) defined by
\begin{equation*}
T_k(a,b):=(ae^{ikt_0}+be^{-ikt_0},ae^{ikt_1}+be^{-ikt_1}).
\end{equation*}
Hence, the estimation \eqref{estmode} holds for all $k$ if and only if all the $T_k$ are invertible and there exists a positive constant $c''$ independent of $k$ such that
\begin{equation*}
\frac{1}{\|T_k^{-1}\|}\ge \frac{c''}{k^{r-s}}.
\end{equation*}
The determinant of $T_k$ equalling $2i\sin k(t_0-t_1)$, we deduce that all the $T_k$ are invertible if and only if $(t_0-t_1)/\pi$ is irrational. In that case, their inverses are given by 
\begin{equation*}
 T_k^{-1}(a,b)=\frac{1}{2i\sin k(t_0-t_1)}(e^{-ikt_1}a-e^{-ikt_0}b,-e^{ikt_1}a+e^{ikt_0}b)
\end{equation*}
and a computation of their norms yield
\begin{equation*}
 \|T_k^{-1}\|=\frac{\sqrt{1+|\cos k(t_0-t_1)|}}{\sqrt{2}|\sin k(t_0-t_1)|}
\end{equation*}
Thus,
\begin{equation*}
\frac{1}{\|T_k^{-1}\|} \asymp |\sin k(t_0-t_1)| \asymp \Big\|\frac{k(t_0-t_1)}{\pi}\Big\|.
\end{equation*}
The first estimation follows from the expression of $\|T_k^{-1}\|$ while the second estimation is a consequence of the Lemma \ref{lemmasinus}. We observe that if \eqref{estthmstring} holds, then $(t_0-t_1)/\pi$ must be irrational and that ensures that all the  $T_k$ are invertible. The proof is complete.
\end{proof}
\begin{thm}${}$\label{MRclassicalstring}
\begin{itemize}
\item[(a)]
If $r-s<1$, there cannot be any strategic pair.
\item[(b)]
If $r-s=1$, 
 the set of strategic pairs  has zero Lebesgue measure  and full Hausdorff dimension in $\R^2$.
\item[(c)]
If $r-s>1$, 
the set of strategic  pairs has full Lebesgue measure in $\R^2$.
\end{itemize}
\end{thm}
In the following lemma, we gather some classical results of Diophantine approximation.
\footnote{
The results concerning the Lebesgue measure are due to A. Khinchin and the result concerning the Hausdorff dimension is due to  V. Jarn\'ik 
}
 For a real number $\alpha$, we set 
$E_{\alpha}:=\{x\in \R : \exists c>0 : \, \|kx\|\ge ck^{-\alpha},\,k=1,2,\ldots\}$.
\begin{lem}[{\cite[pp. 120--121]{Cassels}}, {\cite[p. 104]{Bugeaud}}, {\cite[p. 142]{Falconer}}] \label{lemmadioph}
${}$
\begin{itemize}
\item[(a)]
If $\alpha=1$, then $E_{\alpha}$ has zero Lebesgue measure and full Hausdorff dimension in $\R$.
\item[(b)]
If $\alpha>1$, then $E_{\alpha}$ has full Lebesgue measure in $\R$.
\end{itemize}
\end{lem}

\begin{proof}[Proof of Theorem {\ref{MRclassicalstring}}] The result is a consequence of the Proposition \ref{thmstring} and results of Diophantine approximation. 

If $\alpha<1$ then the set $E_1$ defined in the Lemma \ref{lemmadioph} is empty. Indeed, if we suppose that $x\in E_{\alpha}$, then for sufficiently large $k$, $\|kx\|\ge 1/k$ with is in contradiction with a theorem of Dirichlet (see \cite[p.4]{Cassels}) that asserts that if $x$ is irrational, then the inequality $\|kx\|<1/k$ as infinitely many solutions in $k$.

If $\alpha\ge 1$, then we use the
Lemma \ref{lemmadioph}. One can notice that the set of pairs $(t_0,t_1)\in \R^2$ such that $(t_0-t_1)/\pi \in E_{\alpha}$ has full (resp. zero) Lebesgue measure or Hausdorff dimension in $\R^2$ if $E_{\alpha}$ has full (resp. zero) Lebesgue measure or Hausdorff dimension in $\R$. For the Lebesgue measure, this results from Fubini's theorem. For the Hausdorff dimension, this is a consequence of its behaviour with a product of sets and its invariance by a bi-Lipschitz transformation (see e.g., \cite{Falconer}).
\end{proof}
\pagebreak
\begin{rks}${}$
\begin{itemize}
\item
The assertion (a) of Corollary \ref{MRclassicalstring} can be seen as an \emph{optimality result}. Indeed, it means that with only two observations, the difference $r-s$ between the orders of the Sobolev norms in the inequality \eqref{observation2} must be at least 1.
\item
One cannot obtain such estimations with only one observation. Indeed, let $t_0\in \R$. Then, the function
$
y(t,x)=\sin(t-t_0)\sin(x) 
$
is a solution to \eqref{string} with $y(0)\neq 0$ or $y'(0)\neq 0$, but $y(t_0)=0$. 
\item
If the pair $(t_0,t_1)$ is strategic, then, having only access to the two observations i.e. the position of the string at times $t_0$ and $t_1$, we can recover  the initial data $y_0$ and $y_1$ using the expansion in Fourier series of $y(t_0)$ and $y(t_1)$ and the applications $T_k^{-1}$. Moreover, the observability inequality ensures a ``continuity property'' in this reconstruction process. Indeed, if two sets of observations are close, then the two sets of associated initial data must be close too.
\item
In the same way, if $r-s\ge 1$, we can obtain estimations of the form
\begin{align*}
\|y_0\|_s+\|y_1\|_{s-1}&\le c(\|y'(t_0)\|_{r-1}+\|y(t_1)\|_r),\\
\|y_0\|_s+\|y_1\|_{s-1}&\le c(\|y(t_0)\|_r+\|y'(t_1)\|_{r-1}),\\
\|y_0\|_s+\|y_1\|_{s-1}&\le c(\|y'(t_0)\|_{r-1}+\|y'(t_1)\|_{r-1}),
\end{align*}
\item
Applying the Hilbert Uniqueness Method (see \cite{Lions1988}, \cite{Kom94}), it is possible to prove the following \emph{exact controllability} result : let $0<t_0<t_1<T$ such that the observability inequality \eqref{observation2} holds with $r=0$ and $s=-1$. Then, given initial data $(y_0,y_1) \in D^2\times D^1$, we can find two control vectors $v$ and $w$ in $D^0$ such that the solution (that can be defined rigorously)  of the inhomogeneous problem
\footnote{
$\delta$ is the Dirac mass in $0$.
} 
\begin{equation*}
\begin{cases}
y''-y_{xx}=\delta(t-t_0)v+\delta(t-t_1)w &\text{in }(0,T)\times(0,\pi),\\
y=0&\text{on }(0,T)\times\{0,\pi\},\\
y(0)=y_0,\quad y'(0)=y_1 &\text{in } (0,\pi).
\end{cases}
\end{equation*}
satisfy $y(T)=y'(T)=0$.
\end{itemize}
\end{rks}
\section{With more observations}\label{more}
In this paragraph, we still assume that $q=0$ in \eqref{string}.
In section \ref{classicalstring}, we have seen that with only two observations, it is necessary that $r-s\ge 1$ in the estimation \eqref{observation2}. In this paragraph, we prove that adding other observations, it is possible to reduce the difference $r-s$.
\begin{thm}\label{moreobservations}
Let $t_1, t_2, \ldots, t_n \in \mathbb{R}$ with $n \geq 2$, $r\in\mathbb{R}$ and set $s:=r-1/(n-1)$.
Assume that among the $(t_i-t_j)/\pi$, $1\le i,j\le n$, we can extract $n-1$ elements $\tau_1,\ldots,\tau_{n-1}$ that belong to a real algebraic extension of $\Q$ of degree $n$ and such that 
$1,\tau_1,\ldots,\tau_{n-1}$ are linearly independent over $\Q$.
Then, there exists a positive constant $c$ such that
\begin{equation*}
\|y_0\|_s+\|y_1\|_{s-1}\leq c(\|y(t_1)\|_r+\ldots+\|y(t_n)\|_r)
\end{equation*}
for all initial data $(y_0,y_1)\in D^r\times D^{r-1}$.
\end{thm}
The proof relies on the following
\footnote{
The second inequality in the Lemma will only be used in section \ref{other}.
}
\begin{lem}[{\cite[p. 79]{Cassels}}]\label{lemmaCassels}
Let $x_1,\ldots,x_n$ be numbers that belong to a real algebraic extension of $\Q$ of degree $n+1$ such that $1,x_1,\ldots,x_n$ are linearly independent over $\Q$. Then, there exists a positive constant $c$, only depending on $x_1,\ldots,x_n$, such that
\begin{equation*}
\max\|kx_j\| \geq ck^{-1/n}, \qquad k=1,2,\ldots
\end{equation*}
and
\begin{equation*}
\|k_1x_1+k_2x_2+\ldots k_nx_n\| \geq c(\max |k_j|)^{-n},
\qquad (k_1,\ldots,k_n)\in \Z^n\setminus \{(0,\ldots 0)\}.
\end{equation*}
\end{lem}
\begin{proof}[Proof of Theorem \ref{moreobservations}]
Adapting the method described in the proof of Theorem \ref{thmstring}, it is sufficient to obtain the estimation
\begin{equation*}
 \sum_{p=1}^n|ae^{ikt_p}+be^{-ikt_p}|^2\geq ck^{-\frac{2}{n-1}}(|a|^2+|b|^2),
\end{equation*}
where $c$ is a positive constant, independent of $a,b\in\mathbb{C}$ and $k\in\N^*$. With no loss of generality, we can assume that $\tau_p=(t_1-t_{p+1})/\pi$ for $p=1,\ldots,n-1$. We have
\begin{eqnarray*}
 \sum_{p=1}^n|ae^{ikt_p}+be^{-ikt_p}|^2 &=& \sum_{p=2}^n(\frac{1}{n-1}|ae^{ikt_1}+be^{-ikt_1}|^2+ |ae^{ikt_p}+be^{-ikt_p}|^2)\\
&\geq&c_1  \sum_{p=2}^n(|ae^{ikt_1}+be^{-ikt_1}|^2+|ae^{ikt_p}+be^{-ikt_p}|^2)\\
&\geq& c_2 \Big(\sum_{p=2}^n|\sin k(t_1-t_p|^2\Big)(|a|^2+|b|^2)\\
&\geq& c_3 \Big(\sum_{p=2}^n\Big\|\frac{k(t_1-t_p)}{\pi}\Big\|^2\Big)(|a|^2+|b|^2)\\
&\geq& c_3 \max\Big\|\frac{k(t_1-t_p)}{\pi}\Big\|^2(|a|^2+|b|^2)\\
&\geq& c_4 k^{-2/(n-1)}(|a|^2+|b|^2)
\end{eqnarray*}
for all $k=1,2,\ldots$, with positive constants $c_1,c_2,c_3,c_4$ independent of $a,b\in\mathbb{C}$.
The numbers $1,(t_1-t_2)/\pi,\ldots,(t_1-t_n)/\pi$ are independent over $\Q$. In particular the numbers $(t_1-t_p)/\pi$, $p=1,\ldots,n$ are irrational numbers. This ensures that some corresponding linear transformations on $\C\times\C$ (see the proof of Theorem \ref{thmstring}) are invertible and implies the second inequality. The third inequality is a consequence of Lemma \ref{lemmasinus} while the last inequality results from Lemma \ref{lemmaCassels}.
\end{proof}
\begin{rk}
Formally, letting the number of observations  tend to $+\infty$, setting $r=0$ and $T>0$, we recover an internal observability result :
\begin{equation*}
\|y_0\|^2_0+\|y_1\|^2_{-1}\le c\int_0^T\int_0^{\pi}|y(t,x)|^2\ud x\ud t .
\end{equation*} 
\end{rk}
\section{Observability of the loaded string ($q>0$)}\label{loadedstring}
In this paragraph, we assume that $q>0$ in \eqref{string} and that $r$ and $s$ are two real numbers such that $r-s=1$. First, let us recall the
\begin{thm}[A. Szij\'art\'o and J. Heged\H us, {\cite[Theorem 1 p.4]{SH12}}]\label{thmSH}
Let $t_0$ and $t_1$ be real numbers such that 
\begin{equation}\label{hyp1}
\frac{t_0-t_1}{\pi} \in \Q
\end{equation}
and 
\begin{equation} \label{hyp2}
\sin((t_0-t_1)\sqrt{k^2+q})\neq0,\qquad k=1,2,\ldots
\end{equation}
Then, $(t_0,t_1)$ is an strategic pair.
\end{thm}
Are such hypotheses easily satisfied? We can answer  this question with the following
\begin{prop}\label{density}
The set of strategic pairs satisfying the hypotheses \eqref{hyp1} and \eqref{hyp2} is dense in $\R^2$.
\end{prop}
\begin{proof}
It is sufficient to prove that for each real number $\tau$ and each $\delta>0$, there exists a real number $\tau'$ satisfying the three conditions : $|\tau-\tau'|<\delta$, $\tau' \in \pi\Q$ and $\sin\Big(\tau' \sqrt{k^2+q} \Big) \neq 0$ for all $k=1,2,\ldots$

First, we notice that $\sin(\zeta \sqrt{k^2+q})=0$ if and only if $\zeta\sqrt{k^2+q}\in \pi\Z$. Now, we distinguish three cases.

1. \emph{If $q$ is an irrational number.} The set $\pi\Q$ being dense in $\R$, there exists a number $\tau'\in \pi\Q$ such that $|\tau-\tau'|\le \delta$. Moreover, $\tau'$ can be written $\tau'=(a/b)\pi$
with $a\in \Z$ and $b\in \N^*$ relatively primes. Assume that there exist $k\in \N^*$ and $n\in \Z$ such that
\begin{equation*}
\tau'\sqrt{k^2+q}=n\pi \quad \iff \quad \frac{a}{b}\sqrt{k^2+q}=n.
\end{equation*}
Then,
\begin{equation*}
q=\frac{n^2b^2}{a^2}-k^2\in \Q,
\end{equation*}
which is in contradiction with our assumption on $q$.

2. \emph{If $q$ is an integer}. We recall that if $(a/b)\pi\in \pi\Q$, then, $\sin((a/b)\pi\sqrt{k^2+q})=0$ if and only if $(a/b)\sqrt{k^2+q}\in \Z$.
Moreover, the quantity $\sqrt{k^2+q}$ is either an integer or an irrational number (depending on the fact that $k^2+q$ is a square or not). For sufficiently large $k$, $\sqrt{k^2+q}$ cannot be an integer. Indeed,
\begin{equation*}
\sqrt{k^2+q}=k\sqrt{1+\frac{q}{k^2}}=k\Big( 1+\frac{q}{2k^2}+o\big
(\frac{1}{k^2}\big)\Big)=k+\frac{q}{2k}+o\big(\frac{1}{k}\big)
\end{equation*}
and this is not an integer for sufficiently large $k$. Hence, for such $k$, it is an irrational number and so is $(a/b)\sqrt{k^2+q}$. Now, let $\tau'':=(a/b)\pi\in \pi\Q$ such that
\begin{equation*}
|\tau''-\tau|<\frac{\delta}{2}.
\end{equation*}
We are going to perturb a little bit the rational number $(a/b)$ in order to construct a number $\tau'$ such that the sine neither vanish.  From the above discussion, the quantity $\sqrt{k^2+q}$ can take at most a finite number of integer values when $k$ varies. We denote them by $x_1,\ldots,x_N$ (if it does not take any integer value, then it is always an irrational number and we can take $\tau'=\tau''$). Let $p$ be a prime number that is not a divisor of any of the numbers $x_1,\ldots,x_N$. For sufficiently large $n$, 
\begin{equation*}
\Big|\pi\frac{(p^n-1)a}{p^nb}-\tau\Big|<\delta.
\end{equation*}
and $p^n$ does not divide $a$. Now, two cases are possible. If $\sqrt{k^2+q}$ is not an integer, then it is an irrational number and 
\begin{equation*}
\frac{(p^n-1)a}{p^nb}\sqrt{k^2+q} \not \in \Z.
\end{equation*}
On the other hand, if $\sqrt{k^2+q}$ is an integer, then $\sqrt{k^2+q}=x_l$ for one $l\in \{1,\ldots,N\}$ and
\begin{equation*}
\frac{(p^n-1)a}{p^nb}\sqrt{k^2+q} =\frac{(p^n-1)ax_l}{p^nb}\not \in \Z.
\end{equation*}
because $p^n$ does not divide $(p^n-1)ax_l$. Finally, 
\begin{equation*}
\tau':=\frac{(p^n-1)a}{p^nb}\pi
\end{equation*}
satisfies the three expected conditions.

3. \emph{If $q$ is a rational number but not an integer.} Then, we can write $q=c/d$, where $c$ and $d$ are integers. Hence,
\begin{equation*}
\tau\sqrt{k^2+q}=\tau\sqrt{k^2+\frac{c}{d}}=\tau\sqrt{k^2+\frac{cd}{d^2}}=\frac{\tau}{d}\sqrt{k^2d^2+cd}
\end{equation*}
and we are lead back to the case where $q$ is an integer.
\end{proof}
Now, we give another method to obtain an observability result for the loaded string.
\begin{thm}\label{MRloadedstring}
Let $(t_0,t_1)$ be a strategic pair for the classical string i.e.
\begin{equation}\label{hyp3}
|\sin(k(t_0-t_1)| \ge \frac{c}{k},\qquad k=1,2,\ldots
\end{equation}
for a suitable positive constant $c$.
Then, it is also a strategic pair for the loaded string, provided that $q$ is sufficiently small.
\end{thm}
\begin{rk}
This result can be viewed as a complementary result to Theorem \ref{thmSH} since the hypothesis \eqref{hyp3} implies that $(t_0-t_1)/\pi$ is irrational; hence \eqref{hyp1} cannot hold.
\end{rk}
\begin{proof}
Applying the method described in the proof of Proposition \ref{thmstring}, a necessary and sufficient condition for estimation \eqref{observation2} to hold true is 
\begin{equation}\label{estpot}
|\sin (\omega_k(t_0-t_1))|=|\sin(\sqrt{q+k^2}(t_0-t_1))|\ge \frac{c'}{k},\qquad k=1,2,\ldots,
\end{equation}
where $c'$ is a positive constant, independent of $k$.

 Comparing the quantities $|\sin \omega_k(t_0-t_1)|$ and $|\sin (k(t_0-t_1))|$, we will find a sufficient condition that implies \eqref{estpot}. Let us estimate the difference
\begin{equation*}
|\sin(\sqrt{q+k^2}(t_0-t_1))-\sin( k(t_0-t_1))|.
\end{equation*}
For a fixed $k\in \N^*$, we consider the application $f_k$, defined for $x \ge 0$ by
\begin{equation*}
f_k(x):=\sin(\sqrt{k^2+x}(t_0-t_1)).
\end{equation*}
We have
\begin{align*}
|f_k'(x)|&=\frac{|\cos(\sqrt{k^2+x}(t_0-t_1))||t_0-t_1|}{2\sqrt{k^2+x}}\\
&\le \frac{|t_0-t_1|}{2k}.
\end{align*}
From the triangle inequality and the mean value theorem, 
\begin{equation*}
|f_k(0)|-|f_k(q)|\le|f_k(q)-f_k(0)|\le \frac{|t_0-t_1|q}{2k}.
\end{equation*}
Hence,
\begin{align*}
|\sin(\sqrt{q+k^2}(t_0-t_1))|&\ge |\sin (k(t_0-t_1))|-\frac{|t_0-t_1|q}{2k}\\
&\ge \frac{c}{k}-\frac{|t_0-t_1|q}{2k}
\end{align*}
and these estimations are satisfied for all $k=1,2,\ldots$ Thus, if the quantity 
\begin{equation}\label{positiveconstant}
c':=c-\frac{|t_0-t_1|q}{2}
\end{equation}
is positive, the estimation \eqref{observation2} is true. A sufficient condition is 
\begin{equation}\label{qsmall1}
q<\frac{2c}{|t_0-t_1|}. \qedhere
\end{equation}
\end{proof}
\begin{rks}${}$
\begin{itemize}
\item
The inequality \eqref{qsmall1} can be rewritten more precisely as
\footnote{
$K(x)$ denotes the largest partial quotient in the continued fraction of $x$, i.e. if the development in continued fraction of $x$ is given by $x=[a_0;a_1,a_2,\ldots]$, then $
K(x):=\sup_{k\ge 1}a_k$.
}
\begin{equation*}
q<\frac{4}{|t_0-t_1|(K((t_0-t_1)/\pi)+2)}.
\end{equation*}
Indeed, from  the proof of Lemma \ref{lemmasinus} and classical results of Diophantine approximation (see  \cite{Shallit}), the hypothesis \eqref{hyp3} holds if and only if the the number $(t_0-t_1)/\pi$ is badly approximable by rational numbers so that its partial quotients are bounded i.e. $K((t_0-t_1)/\pi)$ is finite. Moreover,
\begin{equation*}
|\sin k(t_0-t_1)|\ge 2\Big\| k\frac{t_0-t_1}{\pi}\Big\|\ge \frac{2}{(K((t_0-t_1)/\pi)+2)k}.
\end{equation*}
\item
It is possible to avoid a restriction on the size of the potential $q$. Set $\xi:=(t_0-t_1)/\pi \in \R\setminus\Q$ and 
\begin{equation*}
\nu(\xi):=\liminf_{k\to+\infty}k\|k\xi\|.
\end{equation*}
If $\xi$ is badly approximable, then $\nu(\xi)>0$. Moreover,  if $\xi'$ is an irrational number such that its partial quotients coincide with those of $\xi$ from a certain rank, then $\nu(\xi')=\nu(\xi)$ (see \cite[p. 11]{Cassels}). Let us construct a strictly decreasing sequence of irrational numbers by setting $\xi_0=\xi$ and 
\begin{equation*}
\xi_{n+1}= \frac{\xi_n}{1+\xi_n}=\frac{1}{1+1/\xi_n}.
\end{equation*}
We can assume that $0<\xi<1$ so that its development in continued fraction has the form
\begin{equation*}
\xi=\xi_0=[0;a_1,a_2,a_3,\ldots].
\end{equation*}
Therefore, $1/\xi_0=[a_1;a_2,a_3,\ldots]$ and $1+1/\xi_0=[1+a_1;a_2,a_3,\ldots]$, whence $\xi_1=[0;1+a_1,a_2,a_3,\ldots]$ and by recurrence 
\begin{equation*}
\xi_n=[0;n+a_1,a_2,a_3,\ldots].
\end{equation*}
Thus, for all $n$, $\nu(\xi_n)=\nu(\xi)>0$ and the sequence $(\xi_n)$ converges to zero. 
Now, from the definition of $\nu(\xi)$  and the Lemma \ref{lemmasinus}, we obtain, for $k$ \emph{sufficiently large},
\begin{eqnarray*}
|\sin k\pi\xi_n|\ge 2\|k\xi_n\|\ge 2\frac{\nu(\xi)}{k}.
\end{eqnarray*}
Hence, going back to the relation \eqref{positiveconstant}, if we choose $n$ sufficiently large so that 
\begin{equation*}
2\nu(\xi)-\frac{\xi_n\pi q}{2}>0
\end{equation*}
and if we assume moreover that 
\begin{equation*}
\sin(\om_k\pi\xi_n)\neq 0, \qquad k=1,2,\ldots,
\end{equation*}
then, choosing $t_0$ and $t_1$ such that $t_0-t_1=\pi\xi_n$, the observability inequality holds.
\end{itemize}
\end{rks}
\section{Extension of the method to beams and plates}\label{other}
\subsection{Observability of a hinged beam}
The small transversal vibrations of a hinged beam of length $\pi$ satisfy
\begin{equation} \label{beam}
\begin{cases}
y''+y_{xxxx}=0 & \text{in } \R\times (0,\pi),\\
y=y_{xx}=0 & \text{in } \R\times\{0,\pi\},\\
y(0)=y_0, \quad y'(0)=y_1 & \text{in } (0,\pi).
\end{cases}
\end{equation}

Using the same spaces $D^s$ as for the vibrating string, we have the 
\begin{prop}
Let $s \in \R$. For all initial data $y_0 \in D^s$ and $y_1 \in D^{s-2}$, the problem \eqref{beam} admits a unique solution $y \in C(\R,D^s) \cap C^1(\R,D^{s-2}) \cap C^2(\R,D^{s-4})$ given by
\begin{equation} \label{solutionbeam}
 y(t,x)=\sum_{k=1}^{\infty}(a_k e^{ik^2t}+b_k e^{-ik^2t})\sin{kx},
\end{equation}
where the complex coefficients $a_k$ and $b_k$ satisfy
\begin{equation} \label{conservationbeam}
 \|y_0\|_s^2+\|y_1\|_{s-2}^2\asymp\sum_{k=1}^{\infty}k^{2s}(|a_k|^2+|b_k|^2).
\end{equation}
\end{prop}

In this case, the observability problem turns to the following one : given two real numbers $r$ and $s$ such that $s\le r$, we are looking for two times $t_0$ and $t_1$ such that 
\begin{equation} \label{observationbeam}
 \|y_0\|_s+\|y_1\|_{s-2}\leq c(\|y(t_0)\|_r+\|y(t_1)\|_r)
\end{equation}
for a positive constant $c$, independent of the initial data $(y_0,y_1)\in D^r \times D^{r-1}$. Again, such a pair $(t_0,t_1)$ will be called an \emph{strategic pair}.
\begin{prop}
The pair $(t_0,t_1)$ is strategic if and only if there is a positive constant $c$ such that
\begin{equation*}
\Big\|\frac{k^2(t_0-t_1)}{\pi}\Big\| \ge \frac{c}{k^{r-s}}, \qquad k=1,2,\ldots
\end{equation*}
\end{prop}
\begin{thm}${}$\label{MRbeam}
\begin{itemize}
\item[(a)]
If $r-s=2$, 
there is a set of strategic pairs that is infinite, has zero Lebesgue measure  and full Hausdorff dimension in $\R^2$.
\item[(b)]
If $r-s>2$, 
there is a set of strategic pairs that has full Lebesgue measure in $\R^2$.
\end{itemize}
\end{thm}
\subsection{Observability of a hinged rectangular plate}
Let $a$ and $b$ be positive real numbers and $\Omega=(0,a)\times(0,b)\subset\R^2$
the rectangular domain whose boundary is denoted by $\Gamma$. The small transversal vibrations of a hinged plate whose shape is delimited by $\Omega$ satisfy
\begin{equation}\label{plate}
\begin{cases}
y''+\Delta^2 y=0 & \text{in } \R\times\Omega,\\
y=\Delta y=0 & \text{in } \R\times\Gamma,\\
y(0)=y_0,\quad y'(0)=y_1 & \text{in } \Omega.
\end{cases}
\end{equation}

The eigenvalues of the operator $-\Delta$ with Dirichlet boundary conditions are (see e.g., \cite{CourantHilbert1}) 
\begin{equation*}
\lambda_{m,n}=\frac{m^2\pi^2}{a^2}+\frac{n^2\pi^2}{b^2},\qquad m,n=1,2,\ldots
\end{equation*}
with associated eigenfunctions
\begin{equation*}
e_{m,n}(x,y)=\sin\frac{mx\pi}{a}\sin\frac{ny\pi}{b}, \qquad m,n=1,2,\ldots
\end{equation*}
These functions form an orthogonal and dense system in $L^2(\Omega)$. For $s\in\R$, we define $D^s$ as the completion of the vector space spanned by the functions $e_{m,n}$ for the euclidean norm 
\begin{equation*}
\Bigg\|\sum_{m,n=1}^{\infty}c_{m,n}e_{m,n} \Bigg\|^2_s:=\sum_{m,n=1}^{\infty}\lambda_{m,n}^s|c_{m,n}|^2.
\end{equation*}

\begin{prop}
Given $y_0\in D^s$ and $y_1\in D^{s-2}$, the problem \eqref{plate} has a unique solution $y\in C(\R,D^s) \cap C^1(\R,D^{s-2})\cap C^2(\R,D^{s-4})$, whose expansion in Fourier series is
\begin{equation*} 
 y(t,x)=\sum_{m,n=1}^{\infty}(a_{m,n} e^{i \lambda_{m,n}t}+b_{m,n} e^{-i \lambda_{m,n}t})e_{m,n}(x,y),
\end{equation*}
where the complex coefficients $a_{m,n}$ and $b_{m,n}$ satisfy
\begin{equation*} 
\|y_0\|_s^2+\|y_1\|_{s-2}^2\asymp \sum_{m,n=1}^{\infty}\lambda_{m,n}^s(|a_{m,n}|^2+|b_{m,n}|^2).
\end{equation*}
\end{prop}

The observability problem can be stated exactly as in the previous paragraph. In other words, we are looking for pairs $(t_0,t_1)$ satisfying the estimation \eqref{observationbeam}.

From the expression of the eigenvalues,
\begin{equation*}
\lambda_{m,n} \asymp m^2+n^2.
\end{equation*}
Moreover, an adaptation of Lemma \ref{lemmasinus} yields
\begin{equation*}
|\sin\lambda_{m,n}(t_0-t_1)|\asymp\Bigg\| \frac{\lambda_{m,n}(t_0-t_1)}{\pi}\Bigg\|.
\end{equation*}
Hence, setting $\theta_1:=(\pi(t_0-t_1))/a^2$, $\theta_2:=(\pi(t_0-t_1))/b^2$ and $\alpha:=(r-s)/2$, and applying the same method as for the vibrating string, we get the 
\begin{prop}\label{thmplate}
The pair $(t_0,t_1)$ is strategic if and only if there is a positive constant $c$ such that
\begin{equation}\label{estplate}
\|m^2\theta_1+n^2\theta_2\|\ge \frac{c}{(m^2+n^2)^{(r-s)/2}},\qquad m,n=1,2,\ldots
\end{equation}
\end{prop}

We give sufficient conditions for \eqref{estplate} to hold.
 
\emph{First case:  particular domains.} We assume that there exists a positive integer $N$ such that $\theta_1=N\theta_2$ or equivalently
\begin{equation*}
b^2=Na^2.
\end{equation*}
Therefore, setting $\theta:=\theta_2$, the estimation \eqref{estplate} simplifies in
\begin{equation*}
\|(Nm^2+n^2)\theta\|\ge \frac{c}{(m^2+n^2)^{(r-s)/2}},\qquad m,n=1,2,\ldots
\end{equation*}
We have already seen that if $r-s\ge 2$, the above estimation holds for some choices of $t_0$ and $t_1$. More precisely \emph{the Theorem  \ref{MRbeam} remains true in this case.}

\emph{Second (general) case.} It is not always possible to uncouple the expression $m^2\theta_1+n^2\theta_2$ as we did in the first case. Nevertheless, we can use some results on the approximation of linear forms by rationals.
\begin{thm}${}$
\begin{itemize}
\item[(a)]
Assume that $r-s=4$. If $t_0$ and $t_1$ are real numbers such that $\theta_1$ and $\theta_2$ belong to a real algebraic extension of $\Q$ of degree 3 and $1,\theta_1,\theta_2$ are linearly independent over the rationals, then $(t_0,t_1)$ is an strategic pair.
\item[(b)]
Assume that $r-s>4$. Then, almost all (in the sense of the Lebesgue measure) couples $(t_0,t_1)$ are strategic.
\end{itemize}
\end{thm}
\begin{proof}
The assertion (a) is a direct consequence of the Theorem \ref{thmplate} and  the second estimation of the Lemma \ref{lemmaCassels}. the assertion (b) is a consequence of the Theorem \ref{thmplate} and of a generalization of the Lemma \eqref{lemmadioph} (see \cite[p. 24]{Bugeaud}).
\end{proof}

\end{document}